\documentclass{techrep}

\usepackage[mathscr]{euscript}
\usepackage{upgreek}    
\usepackage{parskip}    
\SHORTTITLE{Bayesian OED}
\TITLE{A brief note on the Bayesian D-optimality criterion}

\AUTHORS{Alen~Alexanderian\footnote{North Carolina State University, Raleigh, NC, USA.
         \THANKS{alexanderian@ncsu.edu}{\today}
         This work was supported in part by US National Science Foundation grant \#2111044.}}
\ABSTRACT{We consider finite-dimensional Bayesian linear inverse problems with 
Gaussian priors and additive Gaussian noise models.  The goal of this note is to
present a simple derivation of the well-known fact that solving the Bayesian
D-optimal experimental design problem, i.e., maximizing the expected information
gain, is equivalent to minimizing the log-determinant of posterior covariance
operator. We focus on the finite-dimensional inverse problems.
However, the presentation is kept generic to facilitate 
extensions to infinite-dimensional inverse problems.}

\renewcommand{\vec}[1]{{\mathchoice
                     {\mbox{\boldmath$\displaystyle{#1}$}}
                     {\mbox{\boldmath$\textstyle{#1}$}}
                     {\mbox{\boldmath$\scriptstyle{#1}$}}
                     {\mbox{\boldmath$\scriptscriptstyle{#1}$}}}}

\newcommand{\ip}[2]{\left\langle {#1}, {#2} \right\rangle}
\newcommand{\mat}[1]{\mathsf{\mathbf{#1}}}
\newcommand{\Id}{\mat{I}}
\renewcommand{\S}{\tilde{\mat{S}}}
\newcommand{\GM}[2]{\mathcal{N}\left( {#1}, {#2}\right)}

\newcommand{\obs}{\vec{y}} 
\newcommand{\prcov}{\vec{\Gamma}_{\mathrm{pr}}}
\newcommand{\postcov}{\vec{\Gamma}_{\mathrm{post}}}
\newcommand{\noise}{\vec{\Gamma}_{\!\mathrm{noise}}}

\newcommand{\ave}[2]{\mathsf{E}_{{#1}}\!\left\{ {#2} \right\}}

\newcommand{\trace}{\mathsf{tr}}

\newcommand{\R}{\mathbb{R}}

\newcommand{\F}{\mat{F}}

\newcommand{\Q}{\tilde{\mat{Q}}}
\newcommand{\HM}{\mat{H}}
\newcommand{\HMt}{\tilde{\mat{H}}}

\newcommand{\hilb}{\mathscr{V}}

\newcommand{\dparpr}{\dpar_{\mathrm{pr}}}
\newcommand{\dparpost}{\dpar^\obs_{\mathrm{post}}}
\newcommand{\priorm}{\upmu_\mathrm{pr}}
\newcommand{\postm}{\upmu_{\mathrm{post}}}
\renewcommand{\Xi}{X^{-1}}
\newcommand{\dpar}{\vec{m}}

\newcommand{\avey}[1]{\mathsf{E}_{{\vec{y}|\dpar}}\left\{ {#1} \right\}}
\newcommand{\DKL}{\mathrm{D}_{\mathrm{kl}}}
\newcommand{\tran}{{\mkern-1.5mu\mathsf{T}}}                
\newcommand{\noiseinv}{\noise^{-1}}

\begin{document}
\maketitle

\section{Introduction}
Let $\hilb$ denote the $n$-dimensional Euclidean space.
Suppose $\dpar \in \hilb$ is a parameter vector, which we seek to infer 
from the following linear model 
\begin{equation}\label{equ:linmodel}
    \obs = \F \dpar + \vec{\eta}.
\end{equation}
Here, $\vec{y} \in \R^q$ is observed data,  
$\F: \hilb \to \R^q$ is a linear parameter-to-observable map, and
$\vec{\eta}$ is a random $q$-vector that models measurement noise.
We assume $\vec{\eta}$ a centered Gaussian with covariance matrix $\noise$ and 
that $\vec{\eta}$ and $\dpar$ are independent. 

In a Bayesian formulation, we use a realization of $\obs$ to obtain 
a posterior distribution law for the inversion parameter $\dpar$.
Herein, we consider a Gaussian prior, $\priorm = \GM{\dparpr}{\prcov}$. 
The solution of the present Gaussian linear
inverse problem is the posterior measure, $\postm = \GM{\dparpost}{\postcov}$, with
\begin{equation}\label{equ:postmeas}
    \dparpost(\obs) = \postcov(\F^\tran\noise^{-1}\obs + \prcov^{-1}\dparpr)
\quad
\text{and}
\quad
    \postcov = (\HM + \prcov^{-1})^{-1},
\end{equation}
where $\HM = \F^\tran \noise^{-1} \F$.  Note that, in general, the posterior
measure describes a distribution law for the inversion parameter $\dpar$ that is
conditioned on the observed data and is absolutely continuous with respect to
the prior measure.

Optimal experimental
design~\cite{Pazman86,AtkinsonDonev92,ChalonerVerdinelli95,Ucinski05} concerns
the question of how to collect experimental data $\vec{y}$ so that the parameter
estimation is ``optimal'' in some sense. In this context, an \emph{experimental 
design} (or design for short) specifies the manner in which experimental data is
collected.  For example, the design might be the placement of measurement points
where data is recorded.  The definition of what constitutes an optimal design
leads to the choice of the design criterion.  This note is concerned with
Bayesian D-optimality.  It is well-known that in the Gaussian linear setting,
Bayesian D-optimality (i.e., maximizing the expected information gain) is
equivalent to minimizing the log-determinant of the posterior covariance operator.  
Below, we will present a simple proof of this fact by showing 
\begin{equation}\label{equ:avgKL} \ave{\priorm}{\avey{\DKL(\postm\| \priorm)}} =
-\frac12 \log \det(\postcov) + \mathrm{const}, \end{equation} where $\DKL$
denotes the Kullback--Leibler (KL) divergence~\cite{kullback:1951} from the
posterior measure to the prior measure.

We state the results in an abstract form where the design is not explicitly
incorporated in the formulations. In general, the
design would enter the Bayesian inverse problem through the data likelihood. The exact
nature of such formulations is problem-dependent and can be found in numerous
papers and books on optimal design of experiments; see
e.g.,~\cite{Ucinski05,HaberHoreshTenorio08,AlexanderianPetraStadlerEtAl14,Alexanderian21}.  Moreover, even
though the parameter space is taken as the $n$-dimensional Euclidean space
$\hilb = \R^n$ equipped with the Euclidean inner product, the following 
presentation is kept sufficiently generic to also enable motivating extensions
to an infinite-dimensional  Hilbert space setting.  This is discussed briefly
towards the end of this note. However, all the derivations are done in finite
dimensions to keep the discussion simple. 

\section{Preliminaries}\label{sec:preliminaries}
In our discussion of the concept of information gain in a Bayesian inverse problem,
we will be using the KL divergence between Gaussian measures on 
the $n$-dimensional Euclidean space. We recall that for Gaussian measures 
$\mu_1 = \GM{{\vec{a}_1}}{\mat{\Sigma}_1}$
and $\mu_2 = \GM{{\vec{a}_2}}{\mat{\Sigma}_2}$, the KL divergence from 
$\mu_1$ to $\mu_2$ is given by 
\begin{equation}\label{equ:gauss-kl}
 \DKL(\mu_1\| \mu_2) = \frac{1}{2} 
                      \Big[ -\log\left( \frac{\det\mat{\Sigma}_1}{\det \mat{\Sigma}_2}\right) - n + 
                              \trace(\mat{\Sigma}_2^{-1} \mat{\Sigma}_1) + 
                              (\vec{a}_1 - \vec{a}_2)^\tran \mat{\Sigma}_2^{-1} (\vec{a}_1 - \vec{a}_2)
                      \Big].
\end{equation}
We also recall the definition of the mean and covariance operator of a Borel
probability measure on $\hilb$.  Let $\{\vec{e}_1, \ldots, \vec{e}_n\}$ be the
standard basis in $\hilb$.  The mean $\vec{a}$ and covariance operator
$\mat{C}$ of a Borel probability measure $\mu$ on $\hilb$ are defined through:
\[
a_i = \ip{\vec{e}_i}{\vec{a}} = \int_{\hilb} \ip{\vec{e}_i}{\vec{x}} \, \mu(d\vec{x}) \quad \mbox{ and }\quad
C_{ij} = \int_{\hilb} \ip{\vec{e}_i}{\vec{x}-\vec{a}} \ip{\vec{e}_j}{\vec{x}-\vec{a}} \, \mu(d\vec{x}),
\] 
where $\ip{\cdot}{\cdot}$ denotes the Euclidean 
inner-product and $i, j \in \{1, \ldots, n\}$.\footnote{The definitions of the mean and covariance operator
extend naturally to the case of Borel probability measures on infinite-dimensional
Hilbert spaces~\cite{DaPrato}.} 

In what follows, we need the following well-known result regarding
the expectation of a quadratic form, which we also prove to keep the
presentation self-contained.
\begin{proposition}\label{prp:quadform}
Consider a Borel probability measure $\mu$ on $\hilb$,
with mean $\vec{a}$ and (symmetric positive definite)
covariance matrix $\mat{C}$. 
Then, for a linear operator $\mat{Q}$ on $\hilb$, 
\begin{equation}\label{equ:fdcase}
   \int_{\hilb} \vec{x}^\tran \mat{Q} \vec{x} \, \mu(d\vec{x}) =
   \trace(\mat{Q}\mat{C})
   + \vec{a}^\tran \mat{Q} \vec{a}.
\end{equation}
\end{proposition}
\begin{proof}
First note that 
\[
\begin{aligned}
       \int_{\hilb} \vec{x}^\tran \mat{Q} \vec{x} \, \mu(d\vec{x}) &=  \int_{\hilb} (\vec{x} - \vec{a})^\tran \mat{Q} (\vec{x} - \vec{a})\, \mu(d\vec{x})  + 
                            \int_{\hilb} \big(\vec{a}^\tran \mat{Q} \vec{x} + \vec{x}^\tran \mat{Q} \vec{a}\big)\, \mu(d\vec{x})  - \vec{a}^\tran \mat{Q} \vec{a} \\
       &= \int_{\hilb} (\vec{x} - \vec{a})^\tran \mat{Q} (\vec{x} - \vec{a})\, \mu(d\vec{x}) + \vec{a}^\tran \mat{Q} \vec{a}.
\end{aligned}
\]
The following completes the proof: 
\[
\begin{aligned}
\int_{\hilb} (\vec{x} - \vec{a})^\tran \mat{Q} (\vec{x} - \vec{a}) \, \mu(d\vec{x})
&= 
\int_{\hilb} \sum_{i,j=1}^n Q_{ij}  
  \ip{\vec{e}_i}{\vec{x}-\vec{a}} \ip{\vec{e}_j}{\vec{x}-\vec{a}} \, \mu(d\vec{x}) \\ 
&= \sum_{i,j=1}^n Q_{ij} \int_{\hilb} 
  \ip{\vec{e}_i}{\vec{x}-\vec{a}} \ip{\vec{e}_j}{\vec{x}-\vec{a}} \, \mu(d\vec{x}) 
= \sum_{i,j=1}^n Q_{ij} C_{ij} = \trace(\mat{Q}\mat{C}). \qedhere
\end{aligned}
\]
\end{proof}

\section{Bayesian D-optimality}\label{sec:Dopt}
We define the expected information gain as follows:
\newcommand{\aveDKL}{\overline\DKL}
\[
   \aveDKL(\postm \| \priorm) = \ave{\priorm}{\avey{\DKL(\postm\| \priorm)}}.  
\]
In an OED problem, $\aveDKL$ will be a function of a set of 
experimental design parameters. A
Bayesian D-optimal design problem aims to \emph{maximize} this function over
a set of admissible design parameters.  The goal of this section is to
show that, 
\begin{equation}\label{equ:DKLresult}
   \aveDKL(\postm\| \priorm) = -\frac12 \log \det(\postcov) + \frac12 \log \det(\prcov).
\end{equation}
As mentioned before, the experimental design parameters enter the formulation of
the Bayesian inverse problem through that data-likelihood. The prior covariance
operator does not depend on the design parameters.  Thus, from an OED point of
view, \eqref{equ:DKLresult} implies that maximizing the expected information
gain is equivalent to minimizing $\log \det( \postcov)$.

We begin by setting up some notations and basic definitions. 
First, we note the following bit of algebra regarding the posterior covariance:
\begin{equation}\label{equ:postcov_formulas}
    \begin{aligned}
    \postcov &= (\HM + \prcov^{-1})^{-1} \\
             &= \prcov^{1/2} (\HMt + \Id)^{-1} \prcov^{1/2}\\
             &= \prcov^{1/2} \S \prcov^{1/2},
    \end{aligned}
\end{equation}
where $\HMt = \prcov^{1/2}\HM\prcov^{1/2}$ and $\S =  (\HMt + \Id)^{-1}$.
In what follows, we will also need
a convenient expression for $\dparpost - \dparpr$. We note,
\begin{equation}
\begin{aligned}    
\dparpost - \dparpr &= 
\postcov\F^\tran\noise^{-1}\obs + \postcov\prcov^{-1}\dparpr -\dparpr\\
&=
\postcov\F^\tran\noise^{-1}\obs + \postcov( \prcov^{-1} - \postcov^{-1})\dparpr
\\
&=
\postcov\F^\tran\noise^{-1}\obs - \postcov \HM \dparpr.
\end{aligned}    
\end{equation}

Next, we state and prove 
the following lemma, which is the key to proving~\eqref{equ:DKLresult}.
\begin{lemma}\label{lem:quadform}
Let $\priorm = \GM{\dparpr}{\prcov}$ and $\postm = \GM{\dparpost}{\postcov}$ 
be the prior and posterior measures corresponding to a Bayesian
linear inverse problem with additive Gaussian noise model as in~\eqref{equ:linmodel}. Then,
\[
\ave{\priorm}{\avey{ (\dparpost - \dparpr)^\tran \prcov^{-1} (\dparpost - \dparpr)}} = \trace(\S \HMt).
\] 
\end{lemma}
\begin{proof}
Note that
\begin{equation}\label{equ:ugly}
\begin{aligned}
(&\dparpost - \dparpr)^\tran \prcov^{-1} (\dparpost - \dparpr)    
\\
&= 
(\postcov\F^\tran\noise^{-1}\obs - \postcov \HM \dparpr)^\tran 
\prcov^{-1} (\postcov\F^\tran\noise^{-1}\obs - \postcov \HM \dparpr)\\
&= (\postcov\F^\tran\noise^{-1}\obs)^\tran \prcov^{-1}(\postcov\F^\tran\noise^{-1}\obs)
-2 (\postcov\F^\tran\noise^{-1}\obs)^\tran \prcov^{-1} \postcov \HM \dparpr\\ 
&\quad+ (\postcov \HM \dparpr)^\tran \prcov^{-1} (\postcov \HM \dparpr).
\end{aligned}
\end{equation} 
We consider each of the terms in the right-hand side separately. 
The second and third terms are simpler to deal with. 
The third one is simply 
\begin{equation}\label{equ:third_term_prelim}
\dparpr^\tran \HM \postcov \prcov^{-1} \postcov \HM \dparpr   
= \dparpr^\tran \HM \prcov^{1/2} \S^2 \prcov^{1/2} \HM \dparpr,
\end{equation}
where we have used $\postcov = \prcov^{1/2}\S\prcov^{1/2}$ 
(cf.~\eqref{equ:postcov_formulas})
to 
conclude $\postcov \prcov^{-1} \postcov = \prcov^{1/2} \S^2 \prcov^{1/2}$. 
The operator in the quadratic form in the right-hand side of~\eqref{equ:third_term_prelim}
appears several times below. Therefore, for notational convenience, we denote
\begin{equation}\label{equ:Qdef}
\Q = \HM \prcov^{1/2} \S^2 \prcov^{1/2} \HM.
\end{equation}
Using this notation, 
the third 
term in the right-hand side of~\eqref{equ:ugly} is simply
\begin{equation}\label{equ:third_term}
\dparpr^\tran \HM \postcov \prcov^{-1} \postcov \HM \dparpr   
= \dparpr^\tran \Q \dparpr.
\end{equation}

Next, we consider the second term in the right-hand side of~\eqref{equ:ugly}.  Note that our additive
Gaussian noise model implies $\obs | \dpar \sim \GM{\F \dpar}{\noise}$.
Therefore,
\[
    \begin{aligned}
    \avey{(\postcov\F^\tran\noise^{-1}\obs)^\tran \prcov^{-1} \postcov \HM \dparpr}    
    &= 
    \avey{ \dparpr^\tran \HM \postcov \prcov^{-1} \postcov \F^\tran \noise^{-1} \obs}
    \\
    &= 
    \avey{ \dparpr^\tran \HM  \prcov^{1/2} \S^2 \prcov^{1/2} \F^\tran \noise^{-1} \obs}
    \\
    &= \dparpr^\tran \Q \dpar.
    \end{aligned}
\]
Subsequently, we note
\begin{equation}\label{equ:second_term}
    \ave{\priorm}{\avey{ (\postcov\F^\tran\noise^{-1}\obs)^\tran \prcov^{-1} \postcov \HM \dparpr}} 
    = \dparpr^\tran \Q \dparpr. 
\end{equation}

We continue by considering the first term in the 
right-hand side of~\eqref{equ:ugly}. Taking expectation with respect to data 
yields,
\begin{align}
     &\avey{ (\postcov\F^\tran\noiseinv \obs)^\tran \prcov^{-1} (\postcov\F^\tran\noiseinv\obs)}
     \notag\\
     &\quad= \avey{ \obs^\tran \noiseinv \F \postcov \prcov^{-1} \postcov\F^\tran\noiseinv\obs}
     \notag\\
     &\quad= \avey{ \obs^\tran \noiseinv \F \prcov^{1/2} \S^2 \prcov^{1/2}\F^\tran\noiseinv\obs}
     \notag\\
&\quad= \trace(\noiseinv \F \prcov^{1/2} \S^2 \prcov^{1/2}\F^\tran)
+ (\F \dpar)^\tran \noiseinv \F \prcov^{1/2} \S^2 \prcov^{1/2}\F^\tran\noiseinv (\F \dpar)
\notag\\
&\quad= \trace(\S^2 \HMt)
+ \dpar^\tran \Q \dpar. 
\label{equ:innerexp}
   \end{align}
Note that we have used the formula for the expectation of a quadratic form
(Proposition~\ref{prp:quadform}) in the penultimate step, and have performed
some algebraic manipulations to obtain the final expression.
In particular, for the term involving the trace we used
\[
\trace(\noiseinv \F \prcov^{1/2} \S^2 \prcov^{1/2}\F^\tran) = 
    \trace(\prcov^{1/2} \S^2 \prcov^{1/2}\F^\tran\noiseinv\F) = 
    \trace(\S^2 \HMt).
\]
Next, we consider the expectation of the second term in~\eqref{equ:innerexp} 
with respect to
the prior measure. This is again done using the formula for the
expectation of a quadratic form:
\begin{equation}\label{equ:outerexp}
    \begin{aligned}
\ave{\priorm}{\dpar^\tran \Q \dpar} 
&= \trace(\Q \prcov) + \dparpr^\tran \Q \dparpr 
\\
&= \trace(\HM \prcov^{1/2} \S^2 \prcov^{1/2} \HM \prcov) + \dparpr^\tran \Q \dparpr 
\\
&= \trace(\S^2 \HMt^2) + \dparpr^\tran \Q \dparpr.
\end{aligned}
\end{equation}
Therefore, using~\eqref{equ:innerexp} and~\eqref{equ:outerexp}, we have
\begin{equation}\label{equ:first_term}
\begin{aligned}
&\ave{\priorm}{\avey{ (\postcov\F^\tran\noiseinv \obs)^\tran \prcov^{-1} (\postcov\F^\tran\noiseinv\obs)}}
\\
&\quad= \trace(\S^2 \HMt) + \trace(\S^2 \HMt^2) + \dparpr^\tran \Q \dparpr 
\\
&\quad= \trace(\S \HMt) + \dparpr^\tran \Q\dparpr.
\end{aligned}
\end{equation}
Note that in the last step we have combined the trace terms using
\[
\trace(\S^2 \HMt + \S^2 \HMt^2)  
= \trace\big(\S^2 \HMt (\Id + \HMt)\big)   
= \trace(\S^2 \HMt \S^{-1})                
= \trace(\S \HMt).
\]
Finally, using~\eqref{equ:ugly} along with~\eqref{equ:third_term}, \eqref{equ:second_term}, 
and~\eqref{equ:first_term}, we have 
\begin{multline*}
\ave{\priorm}{\avey{(\dparpost - \dparpr)^\tran \prcov^{-1} (\dparpost - \dparpr)}}\\
= \trace(\S \HMt) + \dparpr^\tran \Q \dparpr
-2 \dparpr^\tran \Q \dparpr
+\dparpr^\tran \Q \dparpr
=\trace(\S \HMt). \qedhere
\end{multline*}

\end{proof}

We are now ready to state and prove the following 
result that establishes~\eqref{equ:DKLresult}.
\begin{theorem}
Let $\priorm = \GM{\dparpr}{\prcov}$ and $\postm = \GM{\dparpost}{\postcov}$
be the prior and posterior measures corresponding to a Bayesian
linear inverse problem with additive Gaussian noise model as in~\eqref{equ:linmodel}. Then,
\[
\aveDKL(\postm \| \priorm) = -\frac12 \log \det(\postcov) + \frac12 \log\det(\prcov).
\]
\end{theorem}
\begin{proof}
Let us consider the expression for the KL divergence 
from the posterior to the prior:
\begin{equation*}
   \DKL(\postm\| \priorm) =  
   \frac{1}{2}
                      \Big[ -\log\left( \frac{\det\postcov}{\det \prcov}\right) - n +
                              \trace(\prcov^{-1} \postcov) +
                              (\dparpost - \dparpr)^\tran \prcov^{-1} (\dparpost - \dparpr)
                      \Big].
\end{equation*}
Note that, 
\[
\begin{aligned}
   2\aveDKL&(\postm \| \priorm) \\ 
   &= \ave{\priorm}{\avey{-\log\left( \frac{\det\postcov}{\det \prcov}\right) - n +
                              \trace(\prcov^{-1} \postcov) +
                              (\dparpost - \dparpr)^\tran \prcov^{-1} (\dparpost-\dparpr)}}\\
&= -\log\left( \frac{\det\postcov}{\det \prcov}\right) - n +
\trace(\prcov^{-1} \postcov) + \ave{\priorm}{\avey{(\dparpost - \dparpr)^\tran \prcov^{-1} (\dparpost-\dparpr)}}.    
\end{aligned}
\]
We have,
\begin{equation}\label{equ:minor}
\trace(\prcov^{-1} \postcov) = \trace(\prcov^{-1} \prcov^{1/2} \S \prcov^{1/2}) = \trace(\S).
\end{equation}
Thus, by Lemma~\ref{lem:quadform} 
\begin{align*} 
- n + &\trace(\prcov^{-1} \postcov) + \ave{\priorm}{\avey{(\dparpost - \dparpr)^\tran \prcov^{-1} (\dparpost-\dparpr)}} 
\\
&= -n + \trace(\S) + \trace(\S \HMt) \\ 
&= - n + \trace(\S (\Id + \HMt)) \\     
&= - n + \trace(\Id)                  
= 0.
\end{align*}
Hence, 
\[
\aveDKL(\postm \| \priorm) = -\frac12  \log\left( \frac{\det\postcov}{\det \prcov}\right)
    = -\frac12 \log\det \postcov + \frac12 \log\det \prcov. \qedhere
\]
\end{proof}

We also mention the following alternative form of the expression for $\aveDKL$:
\begin{equation}\label{equ:DKL_alt}
\aveDKL(\postm \| \priorm) 
= 
\frac12 \log \det( \HMt + \Id).
\end{equation}
This can be obtained via the following simple calculation:
\[
\begin{aligned}
\aveDKL(\postm \| \priorm) = -\frac12  \log\left( \frac{\det\postcov}{\det \prcov}\right) 
 &=  -\frac12 \log \Big( \det(\postcov) \det(\prcov^{-1})\Big) \\
 &=  -\frac12 \log \Big( \det(\prcov^{-1/2} \postcov \prcov^{-1/2})\Big) \\
 &= -\frac12 \log \det( (\HMt + \Id)^{-1})\\
 &= \frac12 \log \det( \HMt + \Id).
\end{aligned}
\]

The expression for $\aveDKL$ given by~\eqref{equ:DKL_alt} is important from
both theoretical and computational points of view.  In the first place, we
point out that~\eqref{equ:DKL_alt} can be extended to an infinite-dimensional
Hilbert space setting. This leads to the infinite-dimensional analogue of the
D-optimality criterion.  See~\cite{AlexanderianGloorGhattas16} for details
regarding formulating the Bayesian D-optimal design problem in
infinite-dimensional Hilbert spaces.  The developments in that article include
derivation of the infinite-dimensional version of the expression for the
KL-divergence from the posterior measure to the prior, which is then used to
derive the expression for the expected information gain in the
infinite-dimensional setting.  Such extensions are relevant for Bayesian linear
inverse problems governed by partial differential equations with
function-valued inversion parameters, which are common in applications.  We
also point out the article~\cite{AlexanderianSaibaba18}, which contains a
derivation of the analogue of~\eqref{equ:DKL_alt} for a discretized version of
an infinite-dimensional Bayesian linear inverse problem. Note that such 
discretizations involve working in a finite-dimensional Hilbert space equipped
with a properly weighted inner product~\cite{bui2013computational}.

The expression~\eqref{equ:DKL_alt} for the expected information gain is also
important from a computational point of view.  In ill-posed inverse problems, $\HMt$
can often be approximated accurately with a low-rank representation. Such approximations can be used
for efficiently computing $\aveDKL$ defined according to~\eqref{equ:DKL_alt}.
See~\cite{AlexanderianSaibaba18} for computational methods for optimizing the
expected information gain in large-scale Bayesian linear inverse problems using
low-rank spectral decompositions or randomized matrix methods.
 
\bibliographystyle{plain}
\bibliography{refs.bib}
\end{document}